\documentclass[11pt]{amsart}
\usepackage{amsmath}
\usepackage{amsfonts}
\usepackage{amssymb}

\newcommand{\cS}{{\mathcal S}}
\newcommand{\cO}{{\mathcal O}}
\newcommand{\cF}{{\mathcal F}}
\newcommand{\cA}{{\mathcal A}}
\newcommand{\cD}{{\mathcal D}}
\newcommand{\mbC}{{\mathbb C}}
\newcommand{\cP}{{\mathcal P}}
\newcommand{\cC}{{\mathcal C}}
\newcommand{\cL}{{\mathcal L}}
\newcommand{\mbZ}{{\mathbb Z}}
\newcommand{\lam}{\lambda}

\newcommand{\G}{\Gamma}
\newcommand{\fH}{\mathfrak H}

\newtheorem{thm}{Theorem}[section]

\newtheorem{prop}{Proposition}[section]
\newtheorem{cor}{Corollary}[section]
\newtheorem{define}{Definition}[section]

\def\adots{\mathinner{\mkern2mu\raise0pt\hbox{.}  
\mkern2mu\raise4pt\hbox{.}\mkern1mu
\raise7pt\vbox{\kern7pt\hbox{.}}\mkern1mu}}

\numberwithin{equation}{section}
\begin{document}
\setlength{\parindent}{0pt}
\setlength{\parskip}{1ex plus 0.5ex minus 0.2ex}

\bibliographystyle{ieeetr}

\title[On CRDAHA and finite general linear and unitary groups]
{On CRDAHA and  finite general linear and unitary groups}

\author{Bhama Srinivasan}
  \address{Department of Mathematics, Statistics, and Computer Science (MC 249)\\
           University of Illinois at Chicago\\
           851 South Morgan Street\\
           Chicago, IL  60607-7045}
  \email{srinivas@uic.edu}
  
  \begin{abstract}
We show a connection between Lusztig induction operators 
in finite general linear and unitary groups and parabolic induction in
cyclotomic rational double affine Hecke algebras. Two applications are given:
an explanation of a bijection result of Brou\'e, Malle and Michel, and
some results on modular decomposition numbers of finite general groups. 
\end{abstract}
 
  \maketitle
 \centerline{ Dedicated to the memory of Robert Steinberg} 
 
 \section{Introduction}
 
 Let $\G_n$ be the complex reflection group $G(e, 1, n)$,
 the wreath product of $S_n$ and $\mbZ/e\mbZ$, where $e > 1$
 is fixed for all $n$. Let $H(\G_n)$ be the
 cyclotomic rational double affine Hecke algebra, or CRDAHA,
 associated with the complex reflection group $\G_n$.
 The representation theory of the algebras $H(\G_n)$
 is  related to the representation theory  
 of the groups $\G_n$, and thus to the modular
 representation theory of finite general linear groups $GL(n,q)$
 and unitary groups $U(n,q)$.
 In this paper we study this connection in the context
 of a recent paper of Shan and Vasserot \cite{SV}.
 In particular we show a connection between Lusztig induction operators 
in  general linear and unitary groups and certain operators in a Heisenberg
algebra acting on a Fock space. We give 
two applications of this result, where $\ell$ is a prime not dividing $q$
and $e$ is the order of $q$ mod $\ell$.
 The first is a connection via Fock space
 between an induction functor
in CRDAHA described in \cite{SV} and Lusztig induction, which gives an
explanation for a bijection  given by Brou\'e, Malle and Michel \cite{BMM}
and Enguehard \cite{E} 
between characters in an $\ell$-block of a finite general linear, unitary
or classical group and characters of a corresponding complex reflection group.
The second is an application to the  $\ell$-modular theory of $GL(n,q)$,
describing some  Brauer characters by Lusztig induction, for large $\ell$.
 
 The paper is organized as follows. In Section \ref{RDAHA} we state
 the results on CRDAHA from \cite{SV} that we need. We introduce the category
 ${\cO}(\G) = \oplus_{n \geq 0}{\cO}(\G_n)$ where ${\cO}(\G_n)$ is 
 the category $\cO$ of $H(\G_n)$.

 In Section \ref{finite}
 we describe the $\ell$-block theory of $GL(n,q)$ and $U(n,q)$.
 The unipotent characters 
 in a unipotent block are precisely the constituents of a Lusztig 
 induced character from an $e$-split Levi subgroup. Complex reflection groups
 arise when considering the defect groups of the blocks.
 
 In Section \ref{Fock}
 we introduce the Fock space and the Heisenberg algebra, and describe the connection
 between parabolic induction in CRDAHA and a Heisenberg algebra action on a
 Fock space given in \cite{SV}. We have a Fock space ${\cF}_{m,\ell}^{(s)}$
 where $m, \ell > 1$ are positive integers and $(s)$ is an $\ell$-tuple of integers.
 In \cite{SV} a functor $a_{\mu}^*$, where $\mu$ is a partition, is introduced
 on the Grothendieck group   $[{\cO}(\G)]$ and is identified with an
 operator $S_{\mu}$  of  a Heisenberg algebra on the above Fock space. 
 
  The case $\ell=1$ is  considered in Section \ref{Fock2}.
 We consider a Fock space with a basis indexed by 
 unipotent representations of general linear or unitary groups. We define
 the action of a Heisenberg algebra on this by a Lusztig induction operator
 $\cL_{\mu}$ and prove that it can be identified with an  operator $S_{\mu}$  defined by
 Leclerc and Thibon \cite{LT1}.  This is one of the main results of
 the paper. It involves using a map introduced by 
 Farahat \cite{F} on the characters of symmetric groups, which appears to be not
 widely known.
 
 In Sections \ref{reflection} and \ref{decn}
 we give applications of this result, using the results of Section \ref{Fock}.
 The first application
 is that parabolic induction $a_{\mu}^*$ in CRDAHA and Lusztig induction 
$\cL_{\mu}$ on general linear or unitary groups can be regarded as operators
arising from equivalent representations of 
the  Heisenberg algebra. This gives an explanation for an observation of 
Brou\'e, Malle and Michel on a bijection between Lusztig induced characters
in a block of $GL(n,q)$ and $U(n,q)$ and characters of a complex reflection
group arising from the defect group of the block.

The second application deals with $\ell$-decomposition
numbers of the unipotent characters of $GL(n,q)$
 for large $\ell$. 
Via the $q$-Schur algebra we can
regard these numbers as arising from the coefficients of a canonical basis
 $G^{-}(\lambda)$ of Fock space, where $\lambda$ runs through all partitions,
in terms of the standard basis. The $G^{-}(\lambda)$  then express the Brauer characters of
$GL(n,q)$ in terms of unipotent characters. The $G^{-}(\lambda)$ are also described as 
$S_{\mu}$, and so we finally get that if 
$\lam = \mu + e\alpha$ where $\mu'$ is $e$-regular, the Brauer character 
parametrized by $\lam$ is in fact a Lusztig induced generalized character.

\section{Notation}

 $\cP, \cP_n, \cP^{\ell}, \cP_n^{\ell}$ denote the set of all partitions,
the set of all partitions of $n \geq 0$, the set of all $\ell$-tuples of partitions, 
and the set of all $\ell$-tuples of partitions of integers $n_1, n_2, \ldots n_{\ell}$
such that $\sum n_i = n$, respectively.

If $\mathcal C$ is an abelian category, we write $[\mathcal C]$ for the 
complexified Grothendieck group of $\mathcal C$.

We write $\lam \vdash n$ if $\lam$ is a partition of  $n \geq 0$. The parts
of $\lambda$ are denoted by $\lbrace \lam_1, \lam_2, \ldots \rbrace$.
If $\lam= \lbrace \lam_i \rbrace$, $\mu= \lbrace \mu_i \rbrace$
are partitions, $\lam + \mu = \lbrace \lam_i + \mu_i \rbrace$
and $e\lam = \lbrace e\lam_i \rbrace$ where $e$ is a positive integer.

\section{CRDAHA, complex reflection groups}\label{RDAHA} 

References for this section are  \cite{SV}, \cite{GL}. 
We use the notation of (\cite{SV}, (3.3), p.967).

Let $\G_n=\mu_{\ell} \wr {\cS}_n$, where $\mu_{\ell}$ is the 
group of $\ell$-th roots of unity in $\mbC$ and ${\cS}_n$ is the 
 symmetric group of degree $n$, so that $\Gamma_n$
is a complex reflection group. 
The representation category of $\G_n$ is denoted by ${\rm Rep}({\mathbb C}\G_n)$.
The irreducible modules in 
${\rm Rep}({\mbC}\G_n)$ are known by a classical construction and denoted by
$\bar{L}_{\lam}$ where $\lam \in {\mathcal P}_n^{\ell}$.
Let $R(\G)= \oplus_{n \geq 0}[{\rm Rep}({\mathbb C}\G_n)]$.

Let $\fH$ be the reflection representation
of $\Gamma_n$ and $\fH^*$ its dual. 
The cyclotomic rational double affine Hecke algebra or CRDAHA
 associated with $\G_n$ is denoted by $H(\G_n)$,
and is the quotient of the smash product of $\mbC\G_n$ and the tensor algebra of 
$\fH \oplus \fH^*$ by certain relations. The definition involves certain parameters
 (see \cite{SV}, p.967) which play a role in the results we quote from \cite{SV},
 although we will not state them explicitly.

The category $\cO$ of $H(\G_n)$ is denoted by ${\cO}(\G_n)$. This is the category 
of $H(\G_n)$-modules whose
objects are finitely generated as $\mbC [\fH]$-modules and are $\fH$-locally nilpotent.
Here $\mbC[\fH]$ is the subalgebra of $H(\G_n)$ generated by $ \fH^*$ .
 Then ${\cO}(\G_n)$ is a highest weight category (see e.g. \cite{RSVV})
 and its standard modules are denoted by $\Delta_{\lam}$ where $\lam \in {\mathcal P}_n^{\ell}$.
Let ${\cO}(\G) = \oplus_{n \geq 0}{\cO}(\G_n)$. This is one of the main objects 
of our study.

We then have a ${\mathbb C}$-linear isomorphism
${\rm spe}: [{\rm Rep}({\mathbb C}\G_n)] \rightarrow [{\cO}(\G_n)]$ given by $[\bar{L}_{\lam}] \rightarrow [\Delta_{\lam}]$.  We will from now on consider $[{\cO}(\G_n)]$
instead of $[{\rm Rep}({\mathbb C}\G_n)]$.

Let $r, m, n \geq 0$. 
For $n,r$ we have a parabolic subgroup $\G_{n,r} \cong \G_n
 \otimes {\cS}_r$ of $\G_{n+r}$ ,
and there is a canonical equivalence of categories 
$\cO(\G_{n,r}) = \cO(\G_n) \otimes \cO({\cS}_r)$.
By the work of Bezrukavnikov and Etingof  \cite{BE} there are induction and restriction 
functors $^{\cO} {\rm Ind}_{n,r}: \cO(\G_n) \otimes \cO(\cS_r)\rightarrow
\cO(\G_{n+r})$ and $^{\cO} {\rm Res}_{n,r}: \cO(\G_{n+r}) \rightarrow
\cO(\G_n) \otimes \cO(\cS_r)$.

If $\mu \vdash r$, Shan and Vasserot (\cite{SV}, 5.1) have defined functors
$A_{\mu,!},
A_{\mu}^*,  A_{\mu, *} $ on ${\cD}^b({\cO}({\G}))$. 

Here we will be concerned with $A_{\mu}^*$, defined as follows.
\begin{align} \label{functor1}
A_{\mu}^*: {\cD}^b({\cO}({\G_n})) \rightarrow {\cD}^b({\cO}({\G_{n+mr}})), \nonumber \\
M \rightarrow ^{\cO} {\rm Ind}_{n,mr}(M \otimes L_{m\mu})  
\end{align}

Then $a_{\mu}^*$ is defined as the restriction to $[{\cO}(\G)]$ of $A_{\mu}^*$.

\section{Finite general linear and unitary groups}\label{finite}

In this section we describe a connection between the block 
theory of $GL(n,q)$ or $U(n,q)$, and complex reflection groups. This was first observed by
Brou\'e, Malle  and Michel \cite{BMM} and Enguehard \cite{E} for arbitrary
finite reductive groups.

Let $G_n=GL(n,q)$ or $U(n,q)$. 
The unipotent characters of $G_n$ are indexed by 
partitions of $n$. Using the description in (\cite{BMM}, p.45) 
we denote the character corresponding to $\lam \vdash n$ of $GL(n,q)$ or
the character, up to sign, corresponding to $\lam \vdash n$ of $U(n,q)$
as in \cite{FS} by $\chi_{\lam}$. 

Let $\ell$ be a prime not dividing $q$ and $e$ the order of $q$ mod $\ell$.
The $\ell$-modular representations of $G_n$ have been studied by various authors
(see e.g. \cite{CE}) since they were introduced in \cite{FS}. The partition of
the unipotent characters of $G_n$ into $\ell$-blocks is described in the 
following theorem from \cite{FS}. This classification depends only on $e$,
so we can refer to an $\ell$-block as  an $e$-block, e.g. in Section \ref{reflection}.

\begin{thm} The unipotent characters $\chi_{\lam}$ and $\chi_{\mu}$ of $G_n$
are in the same $e$-block if and only if the partitions $\lam$ and $\mu$
of $n$ have the same $e$-core.
\end{thm}

There are subgroups of $G_n$ called $e$-split Levi subgroups (\cite{CE}, p.190). In the case
of $G_n=GL(n,q)$ an $e$-split Levi subgroup $L$  is of the form a product of smaller general linear groups
over $F_{q^e}$ and $G_k$ with $k \leq n$. In the case of $G_n=U(n,q)$, 
$L$  is of the form a product of smaller general linear groups or of smaller unitary groups
over $F_{q^e}$ and $G_k$ with $k \leq n$. Then a pair $(L,{\chi}_{\lambda})$ 
is an $e$-cuspidal pair if $L$ is $e$-split of the form a product of copies of tori, all
of order $q^{e}-1$  in the case of $GL(n,q)$, or all of orders $q^e-1$, $q^{2e}-1$  or $q^{e/2}+1$ in
the case of $U(n,q)$ and $G_k$, where $G_k$ has an $e$-cuspidal unipotent character
${\chi}_{\lambda}$ (\cite{BMM}, p.18, p.27; \cite{E}, p.42). Here a character of 
$L$ is $e$-cuspidal if it is not
a constituent of a character obtained by Lusztig induction $R_M^L$
from a proper $e$-split Levi subgroup $M$ of $L$.

The unipotent blocks, i.e. blocks containing unipotent characters, are
classified by $e$-cuspidal pairs up to $G_n$-conjugacy. 
Let $B$ be a unipotent block corresponding to $(L,{\chi}_{\lambda})$.
Then if $\mu \vdash n$, ${\chi}_{\mu} \in B$ if and only if
 $<R_L^{G_n}({\chi}_{\lambda}),{\chi}_{\mu}> \neq 0$. As above, $R_L^{G_n}$
is Lusztig induction.
 
The defect group of a unipotent block is contained in $N_{G_n}(T)$ for a maximal torus $T$
of $G_n$ such that $N_{G_n}(T)/T$ is isomorphic to a complex reflection group 
$W_{G_n}(L, \lambda)= \mbZ_e \wr S_k$ for some $k \geq 1$. Thus the irreducible
characters of $W_{G_n}(L, \lambda)$ are parametrized by ${\mathcal P}_k^e$.

Let $B$ be a unipotent block of $G_n$ and $W_{G_n}(L, \lambda)$ as above.
We then have the following theorem due to Brou\'e, Malle and Michel (\cite{BMM}, 3.2)
and to Enguehard (\cite{E}, Theorem B).

\begin{thm} (Global to Local Bijection for $G_n$)\label{BMM}
 Let $M$ be an $e$-split Levi subgroup containing $L$  and let $W_M(L, \lambda)$ be
defined as above for $M$. Let $\mu$ be a partition, and let $I^M_L$ be the isometry 
mapping the character of $W_M(L, \lambda)$ parametrized by the 
$e$-quotient of $\mu$  to the  unipotent character $\chi_{\mu}$ 
of $M$ (up to sign) which is a constituent of $R_M^{G_n}(\lam)$.
Then we have 
$R_M^{G_n} I^M_L = I^{G_n}_L \ {\rm Ind}_{W_M(L, \lambda)}^{W_{G_n}(L, \lambda)}$.
 \end{thm}

The theorem is proved case by case  for "generic groups", 
and thus for finite reductive groups.
We have stated it only for $G_n$.

We state a refined version of the theorem involving CRDAHA
and prove it in Section \ref{reflection}.

\section{Heisenberg algebra, Fock space}\label{Fock}

Throughout this section we use the notation of (\cite{SV},4.2, 4.5, 4.6).

The affine Kac-Moody algebra  $\widehat{{\mathfrak s \ell}_{\ell}}$ is generated by
elements $e_p, f_p, p=0, \ldots \ell -1$, satisfying Serre relations (\cite{SV}, 3.4).
We have $\widehat{{\mathfrak s \ell}_{\ell}}= {\mathfrak s \ell}_{\ell} \otimes \mbC[t,t^{-1}] $.

The Heisenberg algebra is the Lie algebra $\fH$  generated by 
$1, b_r, b'_r$, $r \geq 0$ with relations
$[b'_r, b'_s]= [b_r, b_s] =0$, $[b'_r, b_s] = r1\delta_{r,s}$, $r,s \geq 0$ 
(\cite{SV}, 4.2). In $U({\fH})$ we then have 
elements $b_{r_1}b_{r_2}\ldots $ with $\sum_i r_i=r$. If $\lam \in {\mathcal P}$ we then have the element $b_{\lam}=b_{\lam_1}b_{\lam_2}\ldots$, and then
for any symmetric function $f$ the element
$b_f= \sum_{\lam \in {\mathcal P}} z_{\lam}^{-1}
< P_{\lam},f >b_{\lam}$. Here $P_{\lam}$ is a power sum symmetric function
and $z_{\lam}= \prod_ii^{m_i}m_i!$ where $m_i$ is the number of parts of $\lam$
equal to $i$. The scalar product $ < , >$ is the one used in symmetric functions,
where the Schur functions form an orthonormal basis (see\cite{M}).

We now define Fock spaces ${\cF}_m$, ${\cF}_{m,\ell}^{(d)}$ and ${\cF}_{m,\ell}^{(s)}$,
where $m > 1$.
We first define the Fock space, ${\cF}_{m}$ of rank $m$ and level $1$ (\cite{SV}, 4.5).
Choose a basis 
$(\epsilon_1, \ldots \epsilon_m)$ of ${\mbC}^m$.
The Fock space ${\cF}_{m}$ is  defined as the space 
of semi-infinite wedges of the $\mbC$-vector space ${\mbC}^m \otimes \mbC[t,t^{-1}]$. 
If $d \in \mbZ$, let ${\cF}_{m}^{(d)}$ be the subspace of  elements of the form
$u_{i_1}\wedge u_{i_2} \ldots $, $i_1 > i_2 \ldots $, 
where $u_{i-jm}=\epsilon_i \otimes t^j$, $i_k=d-k+1$ for $k \gg 0$. 
Then ${\cF}_{m}= \oplus_{d \in {\mathbb Z}} {\cF_{m}}^{(d)}$.
If we set $|\lambda,d>= u_{i_1}\wedge u_{i_2} \ldots $, $i_k=\lam_k+d-k+1$, the elements
$|\lambda,d>$ with $\lambda \in \cP$ form a basis of ${\cF_{m}}^{(d)}$.
Then $\widehat{{\mathfrak s \ell}_m}$ acts on ${\cF_{m}}^{(d)}$. This 
set-up has been studied by Leclerc and Thibon (\cite{LT1}, \cite{LT2}).

Similarly choose a basis 
$(\epsilon_1, \ldots \epsilon_m)$ of ${\mbC}^m$  and a basis
$({\epsilon_1}^{'}, \ldots {{\epsilon}_{\ell}}^{'})$ of ${\mbC}^{\ell}$. 
 The Fock space ${\cF}_{m,\ell}$ of rank $m$ and level $\ell$
is  defined as the space of semi-infinite wedges, i.e. elements of the form
$u_{i_1}\wedge u_{i_1} \ldots $, $i_1 ,i_2 \ldots $, where the $u_j$ are vectors in
  a $\mbC$-vector space 
${\mbC}^m \otimes {\mbC}^{\ell} \otimes  {\mbC}[z, z^{-1}]$
given by  $u_{{i+(j-1)}m-km\ell}=\epsilon_i \otimes {\epsilon}_j^{'}\otimes z^k$,
with $i=1,2, \ldots m$, $j=1,2, \ldots \ell$, $k \in \mbZ$. 
Then $\widehat{{\mathfrak s \ell}_{\ell}}$,   $\widehat{{\mathfrak s \ell}_m}$
and $\fH$ act on the space (\cite{SV}, 4.6). 

Let $d \in {\mathbb Z}$.  We have a decomposition
${\cF}_{m,\ell}= \oplus_{d \in {\mathbb Z}} {\cF}_{m,\ell}^{(d)}$ 
defined using semi-infinite wedges, as in the case of
${\cF}_{m}$. 
Then ${\cF}_{m,\ell}^{(d)}$ can be identified 
with the space ${\Lambda}^{d+{\infty}/2}$ defined by Uglov (\cite{U}, 4.1).
This space has a basis which Uglov indexes by  $\mathcal{P}$
or by pairs $(\lambda, s)$ where 
$\lambda\in \mathcal{P}^m$ and $s= (s_p)$ is an $m$-tuple of integers
with $\sum_p= s_p=d$. 
There is a bijection between the two index sets given by 
$\lambda \rightarrow ({\lambda}^*, s)$ where ${\lambda}^*$
is the $m$-quotient of $\lambda$ and $s$ is a particular labeling
of the $m$-core of $\lambda$ (\cite{U}, 4.1, 4.2).

 There is a subspace ${\cF}_{m,\ell}^{(s)}$ of ${\cF}_{m,\ell}^{(d)}$,
 the Fock space associated with $(s)$, 
which is a weight space for the $\widehat{{\mathfrak s \ell}_{\ell}}$  
action (\cite{SV}, p.982). We have 
${\cF}_{m,\ell}^{(d)}= \oplus {\cF}_{m,\ell}^{(s)} $, the sum of weight spaces.
Here we can define a basis $\lbrace|\lambda,s> \rbrace$ with 
$\lambda \in {\cP}^{\ell}$ of ${\cF_{m,\ell}}^{(s)}.$
 The spaces  $  {\cF}_{m,\ell}^{(s)} $ were also
studied by Uglov. 

The endomorphism of ${\mbC}^m \otimes \mbC[t,t^{-1}]$ induced by multiplication by
$t^r$ gives rise to a linear operator $b_r$ and its adjoint $b_r'$ on 
${\cF_{m}}^{(d)}$, and thus to an action of $\fH$ on ${\cF_{m}}^{(d)}$.
We also have an action of $\fH$ by operators $b_r, b_r'$ on  ${\cF_{m, \ell}}^{(s)}$,
and this is the main result that we need (\cite{SV}, p.982).

 We now choose a fixed $\ell$-tuple $s$. With suitable parameters of $H(\G_n)$, each $n$, 
the $\mbC$-vector space
$[\cO(\G)]$ is then canonically isomorphic to  ${\cF}_{m,\ell}^{(s)} $.
We then have the following $\mbC$-linear isomorphisms (\cite{SV}, p.990, 5.20):

\begin{align}
[\cO(\G)] \rightarrow R(\G) \rightarrow {\cF}_{m,\ell}^{(s)}, \nonumber \\
\Delta_{\lam}  \rightarrow \bar{L}_{\lam} \rightarrow  |\lam, s >.
\end{align}

Consider the Fock space ${\cF}_{m,\ell}^{(s)}$ with basis indexed by 
$\lbrace|\lam, s>\rbrace$ where $\lam \in {\mathcal P}^{\ell}$.
The element $b_{s_{\mu}} \in \fH$, i.e. $b_f$ where $f=s_{\mu}$, a Schur 
function, acts by an operator $S_{\mu}$ on the space.
The functor $a_{\mu}^*$ on $[{\cO}(\G)]$ (see Section 3)
is now identified with $S_{\mu}$ by (\cite{SV}, Proposition 5.13, p.990).

Remark. The bijection between $m$-core partitions and the $m$-tuples $(s)$
as above has been studied by combinatorialists (see e.g. \cite{GKS}).

\section{Fock space revisited}\label{Fock2}

References for the combinatorial definitions in this section are \cite{LT1}, \cite{LT2}.
  Given a partition $\mu$ we introduce three operators
on Fock space: an operator $S_{\mu}$ defined by Leclerc and Thibon \cite{LT1},
an operator ${{\cF}}^*(\phi_{\mu})$ defined by Farahat \cite{F} on 
representations of the symmetric groups ${\cS}_n$, 
and the operators ${\mathcal L}_{\mu}$  of Lusztig induction on $G_n$.
The algebra of symmetric functions in $\lbrace x_1, x_2, \ldots \rbrace$
is denoted by $\Lambda$.

The integers $\ell, m$ in Section \ref{Fock} will now be replaced by 
a positive integer $e$ which was used in the context of blocks of $G_n$. Thus
$\G_n=\mu_e\wr {\cS}_n$. 


First consider the space ${\cF}_e^{(d)} $ where $d \in \mbZ$,
 with basis elements $\lbrace|\lam, d>\rbrace$ where $\lam \in {\mathcal P}$.
Leclerc and Thibon \cite{LT1} introduced elements in $U(\fH)$ which we write 
in our previous notation as $b_{h_{\rho}}$ and $b_{s_{\mu}}$, 
acting as operators $V_{\rho}$  and $S_{\mu}$ on ${\cF}_e^{(d)} $
where $\rho, \mu \in \mathcal P$ and $h_{\rho}$ is a homogeneous 
symmetric function.  These operators have a  combinatorial description
as follows. Here we will write $|\lam>$ for $|\lam, d>$.
 
First they define commuting operators $V_k$,
($ k \geq 1$  on ${\cF}_e^{(d)} $ defined by

\begin{align}
 V_k(|\lam>)= \sum_{\mu}(-1)^{-s(\mu/\lambda)}{|\mu>},
\end{align} 
where the sum is over all $\mu$ such that $\mu/\lambda$ is a horizontal
$n$-ribbon strip of weight $k$, and $s(\mu/\lambda)$ is the "spin"
of the strip.

 Here a ribbon is the same as a rim-hook, i.e. a skew-
partition which does not contain a $2 \times 2$ square. The head of the ribbon
is the upper right box and the tail is the lower left
box. The spin is the leg length of the ribbon, i.e. the
number of rows $-1$.

\begin{define}(see \cite{L}) A horizontal $n$-ribbon strip of weight $k$ is a tiling of 
a skew partition by $k$ $n$-ribbons such that the topright-most square of every ribbon
touches the northern edge of the shape. The spin
of the strip is the sum of the spins of all the ribbons.
\end{define}

It can be shown that a tiling of a skew partition as above is unique.
More generally we can then define $V_{\rho}$ where $\rho$ is
a composition. If $\rho= \lbrace \rho_1,\rho_2, \ldots \rbrace$ then
$V_{\rho}= V_{\rho_1}.V_{\rho_2}\ldots$.
Finally we define operators  $S_{\mu}$ acting on ${\cF}_e^{(d)} $ 
which we connect to Lusztig induction.

\begin{define}
$S_{\mu}=\sum_{\rho} {\kappa}_{\mu\rho}V_{\rho}$ where the
${\kappa}_{\mu\rho}$ are inverse Kostka numbers (\cite{LT1}, p.204),
(\cite{L}, p.8). 
\end{define}

{\bf Remark.} Let $p_e(f)$ denote the plethysm by the power function in
$\Lambda$, i.e. $p_e(f(x_1, x_2, \ldots))=
f(x_1^e, x_2^e, \ldots)$. (This is related to a Frobenius morphism;
see \cite{LT2}, p.171.) 
In fact in \cite{LT1} $\fH$ is regarded as a
$\mbC(q)$-space where $q$ is an indeterminate. 
Then $V_{\rho}$ and $S_{\mu}$ are $q$-analogs
of multiplication by $p_e(h_{\rho})$ and $p_e(s_{\mu})$ in $\Lambda$.

Next, let ${\mathcal A}_n$ be the
category of unipotent representations of $G_n$. Let ${\mathcal A} =
\oplus_{n \geq 0} [{\mathcal A}_n].$ We recall from Section 4 that
the unipotent characters of  $G_n$ are denoted by $\lbrace \chi_{\lam}\rbrace$ where $\lam \vdash n$. 
 We now regard $\mathcal A$ as having a basis  $[ \chi_{\lambda}] $ where
$\lambda$ runs through all partitions. 
Then ${\mathcal A}$ is isomorphic to  ${\cF}_e^{(d)} $ as a
$\mbC$-vector space, since ${\mathcal A}$ also has
a basis indexed by partitions. 


We now define Lusztig operators ${\mathcal L}_{\mu}$ on $\mathcal A$
and then relate them to the $S_{\mu}$.

\begin{define} Let $\mu \vdash k$. 
The Lusztig map ${\mathcal L}_{\mu}: \mathcal A \rightarrow \mathcal A$ is
as follows. Define ${\mathcal L}_{\mu}: [{\mathcal A}_n]\rightarrow [{\mathcal A}_{n+ke}]$
 by $[\chi_{\lambda}] \rightarrow [R^{G_{n+ke}}_L(\chi_{\lambda}
\times \chi_{\mu})]$, where $L=G_n \times GL(k, q^e)$ or $L=G_n \times U(k, q^e)$,
 an $e$-split Levi subgroup  of $G_{n+ke}$. 
\end{define}

Finally, consider the characters of ${\cS}_n$.
We denote the character  corresponding
to $\lam \in {\mathcal P}_n$ as $\phi_{\lam}$. We also use $\lambda \in {\mathcal P}_n$
to denote representatives of conjugacy classes of ${\cS}_n$.
Let ${\cC}_n$ be the category of representations of $\cS_n$
and $\cC=\oplus_{n \geq 0} [{\cC}_n]$. 

\medbreak\noindent 
Given  partitions $\nu \vdash (n+ke)$, $\lam \vdash n$ such that $\nu / \lam$
 is defined, Farahat \cite{F} 
has defined a character $\hat{\phi}_{\nu/\lam}$ of $S_k$, as follows.
Let the $e$-tuples $(\nu^{(i)})$, $(\lam^{(i)})$ be the $e$-quotients 
of $\nu$ and $\lam$. Then $\epsilon \prod_i{\phi}_{({\nu^{(i)}}/{\lam^{(i)}})}$,
where $\epsilon = \pm 1$ 
is a character of a Young subgroup of $S_k$, which induces up to the
character $\hat{\phi}_{\nu/\lam}$ of  $S_k$.


We will instead use an approach of Enguehard (\cite{E}, p.37) which is more 
conceptual and convenient for our purpose.

\begin{define} Let $\mu \vdash k$.
The Farahat map $\cF: [\cC_{ek}] \rightarrow [\cC_{k}]$
is defined by $(\cF \chi)(\mu)=\chi(e\mu)$ . It is then extended
to the map $\cF :  [\cC_n] \rightarrow  [\cC_{k}] \times[\cC_{\ell}]$, 
by first restricting a representation of $S_n$ to
$S_{ek}\times S_{\ell}$ and then applying $\cF$ to $[\cC_{ek}]$
and the identity to $[\cC_{\ell}]$.
\end{define}

Fix $\mu \vdash k$. Taking adjoints 
and denoting  ${{\cF}}^*$ by ${{\cF}}^*(\phi_{\mu})$ we then have,
for $\lam \vdash n$:

\begin{define} 
 ${{\cF}}^*(\phi_{\mu}):[\cC_{n}] \rightarrow [\cC_{n+ek}],
 \ \
 \phi_{\lam} \rightarrow 
 {\rm Ind}_{\cS_{ek} \times \cS_{n}}^{\cS_{n+ek}}
(\cF( \phi_{\mu}) \times \phi_{\lam})$. 
\end{define}

By the standard classification of maximal tori in $G_n$ we can denote
a set of representatives of the $G_n$-conjugacy classes of the tori by
$\lbrace T_w  \rbrace$, where $w$ runs over a set of representatives
for the conjugacy classes of $S_n$.
We then have that the unipotent character 
$\chi_{\lambda}= {1\over|S_n|}\sum_{w \in S_n} {\lambda}(w) R_{T_w}^{G_n}(1) $
(see e.g. (\cite{FS}, 1.13).
Here, as before, $R_{T_w}^{G_n}(1) $ is Lusztig induction.

We assume in the proposition below  that when $G_n=U(n,q)$ that
$e \equiv 0\ ({\rm mod}\ 4)$. This is the case that is analogous to 
the case of $GL(n,q)$. The other cases for $e$ require some straightforward
modifications which we mention below.
The proof of the  proposition has been sketched by Enguehard (\cite{E},p.37)
when $G_n=GL(n,q)$. 

Let $M$ be the $e$-split 
Levi subgroup of $G_n$ isomorphic to $GL(k,q^e) \times GL_{\ell}$.
 We denote by $^*R_M^{G}$ the adjoint of the Lusztig map 
$R_M^{G}$. It is an analogue of the map ${{\cF}}^*$, 
 and this is made precise below.

 Given $\lambda \vdash n$, we have a bijection 
$ \phi_{\lam} \leftrightarrow \chi_{\lambda} $ 
between  $[{\cC}_n]$ and  $[\cA_n]$. We then have an
obvious bijection 
$\psi: \phi_{\lam} \leftrightarrow \chi_{\lambda} $ 
between $\cC$ and  ${\mathcal A}$.

\begin{prop} Let $G=GL(ek,q)$ or $U(ek,q)$. In the case of $U(ek,q)$
we assume $e \equiv 0 \ ({\rm mod}\ 4)$.
Let $M \cong GL(k,q^e)$, a subgroup
of $G$. Let  $\psi: \phi_{\lam} \leftrightarrow \chi_{\lambda} $ 
between $\cC$ and  ${\mathcal A}$ be as above.

Then 
(i) If $\lambda \vdash ek$,
$\psi({\cF}(\phi_{\lambda}))= ^*R_M^{G}(\chi_{\lambda})$

(ii) If $\mu \vdash k$, 
$\psi({\cF}^*(\phi_{\mu}))= R_M^{G}(\chi_{\mu})$= $ {\mathcal L}_{\mu} $.
\end{prop}

\begin{proof}
We have $\psi({\cF}(\phi_{\lambda}))= 
{1\over{|{\cS}_k|}} \sum_{w \in {\cS_k}}
(\cF\phi_{\lambda})(w)R^M_{T_w}(1)
={ 1\over{|{\cS}_k|}} \sum_{w \in {\cS_k}}
\phi_{\lambda}(ew)R^M_{T_w}(1)$.

Since the torus parametrized by $w$ in $M$ is parametrized by $ew$ in $G$,
we can write this as ${1\over{|{\cS}_k|}} \sum_{w \in {\cS_k}}\phi_{\lambda}(ew)R^M_{T_{ew}}(1)$.

On the other hand, we have (see \cite{FS}, Lemma 2B), using the parametrization
of tori in $M$,
$^*R_M^{G}(\chi_{\lambda})={ 1\over{|{\cS}_k|} }\sum_{w \in {\cS_k}}\phi_{\lambda}(w)R^M_{T_w}(1)$.
This proves (i). Then (ii) follows by taking adjoints.
\end{proof}

The proposition clearly generalizes to the subgroup $M \cong GL(k.q^e)\times G_{\ell}$
of $G_n$ where $n=ek+\ell$.
In the case of $U(n.q)$, if $e$ is odd we replace $e$ by $e'$ where $e'=2e$ 
with $M \cong GL(k,q^{e'})$, and if $e \equiv 2 \ {\rm mod}\ 4 $
by $e'$ where $e'=e/2$ with $M \cong U(k,q^{e'})$, the proof being similar.

Using the isomorphism between the spaces ${\mathcal A}$,   
${\mathcal C}$ and ${\cF}_e^{(d)} $, we now regard the operators
 $ {\mathcal L}_{\mu} $, ${\cF}^*(\phi_{\mu})$ and $ {\mathcal S}_{\mu} $
as acting on ${\cF}_e^{(d)} $.

We now prove one of the main results in this paper.
\begin{thm}\label{Lusztig}
The operators ${\mathcal L}_{\mu}$ and $S_{\mu}$ on ${\cF}_e^{(d)} $
coincide.
\end{thm}

\begin{proof}
We have shown above that ${\cF}^*(\phi_{\mu})= {\mathcal L}_{\mu}$.
We will now show that ${{\cF}}^*(\phi_{\mu})= S_{\mu}$.

More generally we consider the character $\hat{\phi}_{\nu/\lam}$ of $S_k$
defined by Farahat, where $\nu \vdash (n+ke)$ and $\mu \vdash n$,
and describe it using $\cF$.  
 The restriction of $\phi_{\nu}$ to
$S_n \times S_{ke}$ can be written as a sum of $ \phi_{\lambda} \times \phi_{\nu/\lam}$
where $\phi_{\nu/\lam}$ is a (reducible) character of $S_{ke}$, and characters
not involving $\phi_{\lambda}$. We then define $\hat{\phi}_{\nu/\lam}= \cF({\phi}_{\nu/\lam})$,
a character of $S_k$. We then note that $\hat{\phi}_{\nu/\lam}(u)= {\phi}_{\nu/\lam}(eu)$.
Using the characteristic map we get a corresponding skew symmetric function 
$ s_{\nu^*/\lam^*}$. This is precisely the function which
has been described in (\cite{M}, p.91), since it is derived from the usual 
symmetric function $s_{\nu/\lam}$ by taking $e$-th roots of variables.
 Using the plethysm function $p_e$ and its adjoint $\psi_e$ (\cite{LLT},p.1048)) 
 there we get  $s_{\nu^*/\lam^*}=\psi_e(s_{\nu/\lam})$.

By the above facts we get
$$(\hat {\phi}_{\nu/\lambda},{\phi}_{\mu})$$
 $$ =   (s_{\nu^*/\lam^*}, s_{\mu})  $$                             
$$= ({\psi_e}(s_{\nu/\lambda}), s_{\mu})  $$
$$=(s_{\nu/\lambda},p_e(s_{\mu})),$$
$$=(p_e(s_{\mu}).s_{\lambda}, s_{\nu})$$
$$=(S_{\mu}[ \chi_{\lambda}],[ \chi_{\nu}] ).$$ 

The last equality can be seen as follows. There is a
$\mbC$-linear isomorphism between the algebra 
 $\Lambda$ and ${\cF}_e^{(d)} $, since both have bases 
indexed by $\cP$. Under this isomorphism multiplication by
the symmetric function $p_e(s_{\mu})$ on $\Lambda$
corresponds to the operator $S_{\mu}$ on Fock space 
 (see \cite{LT1}, p.6).

This proves that
${\mathcal L}_{\mu}=S_{\mu}$. 
\end{proof}

We recall that ${\widehat {sl}_e}$ acts on ${\cF}_e^{(d)} $ and hence
on $\mathcal A$.

\begin{cor} The highest weight vectors
$V_{\rho}\emptyset$ of the irreducible components of the
${\widehat {sl}_e}$-module $\mathcal A$ (\cite{LLT},p.1054) can
be described by Lusztig induction.
\end{cor}

{\bf Remark.} In fact Leclerc and Thibon also have a parameter $q$ in their 
definition of $S_{\mu}$, since they deal with a deformed Fock space.
Thus $S_{\mu}$ can be regarded as a quantized
version of a Lusztig operator ${\mathcal L}_{\mu}$.

{\bf Remark.} In the notation of (\cite{LT2}, p.173) we have

$(s_{\nu^*/\lam^*, s_{\mu})=(s_{\nu_0/\lambda_0}s_{\nu_1/\lambda_1} \ldots
s_{\nu_{e-1}/\lambda_{e-1}},s_{\mu}}= c^{\mu}_{\nu/\lambda}$, where
the $c^{\mu}_{\nu/\lambda}$ are Littlewood-Richardson coefficients.
We now have $(\chi_{\nu}, R_M
^{G_n}({\chi}_{\lambda}\times
{\chi}_{\mu}))= {\epsilon}c^{\mu}_{\nu/\lambda} $, where
$\epsilon = \pm 1$.  In particular $c^{(k)}_{\nu/\lambda}$ 
is the number of tableaux of
shape $\nu$ such that $\nu/\lambda$ is a horizontal $e$-ribbon of
weight $k$. Thus the Lusztig operator ${\mathcal L}_k$ can
be described in terms of $e$-ribbons of weight $k$, similar to the
case of $k=1$ which classically is described by $e$-hooks.

\medbreak\noindent
\section{ CRDAHA and Lusztig induction}\label{reflection}

The main reference for parabolic induction in this section is \cite{SV}.

In this section we show a connection between the parabolic induction functor
$a_{\mu}^*$ on $[{\cO}(\G)]$
and the Lusztig induction functor ${\mathcal L}_{\mu}$ in $\cA$ using Fock space. 
In particular this gives an explanation
of the Global to Local Bijection for $G_n$ given in Theorem \ref{BMM}. This can
be regarded as a local, block-theoretic version of Theorem \ref{Lusztig}.

As mentioned in Section \ref{finite},
the unipotent characters $\chi_{\lam}$ in an $e$-block of $G_n$ are constituents of 
the Lusztig map $R^{G_n}_L(\lam)$ where $(L, \lam)$ is an $e$-cuspidal pair.
Up to sign, they are in bijection with the characters of $W_{G_n}(L, \lam)$,
and they all have the same $e$-core.

For our result we can assume $d=0$, which we do from now on.
We set $\ell= m = e$ as in Section \ref{Fock2}. 
We have spaces ${\cF}_e^{(0)}$ and  
${\cF}_{e,e}^{(0)}= \oplus_s{\cF}_{e,e}^{(s)}$ 
 where $s= (s_p)$ is an $e$-tuple of integers with $\sum_p s_p=0$. 
 We now fix such an $s$. 
 
 By (\cite{SV}, 6.17, 6.22, p.1010) we  have a  $U(\fH)$-isomorphism
 between ${\cF}_e^{(0)}$ and ${\cF}_{e,e}^{(0)}$. Let ${\cF}_e^{(s)} $ be
 the inverse image of ${\cF}_{e,e}^{(s)}$ under this isomorphism.
 We then have $\mbC$-isomorphisms from ${\cF}_{e,e}^{(s)}$ to $[\cO(\G)]$,  
and from ${\cF}_{e}^{(s)}$ to  ${\cA}^{(s)}$, where ${\cA}^{(s)}$ is the 
 the subspace of $\mathcal A$ spanned by $[\chi_{\lam}]$ where the 
 $\chi_{\lam}$ are in an $e$-block parametrized by the $e$-core
 labeled by $(s)$ (see Section \ref{Fock}).

 The spaces ${\cF}_e^{(s)} $, ${\cF}_{e,e}^{(s)}$,  $[\cO(\G)]$,
${\cA}^{(s)}$ have  bases  $\lbrace | \lambda, s >: \lam \in \cP \rbrace $, 
 $\lbrace| \lambda, s > : \lam \in {\cP}^e \rbrace$, 
 $\lbrace{\Delta}_{\lambda}:\lam \in {\cP}^e\rbrace$
 and   $[{\chi}_{\lam}]$ where $\lambda$ has $e$-core labeled by $s$,
respectively. 
 
We have maps $S_{\mu} : {\cF}_{e,e}^{(s)} \rightarrow {\cF}_{e,e}^{(s)}$,
$\mu \in {\cP}^e$,
 $S_{\mu} : {\cF}_e^{(s)}\rightarrow  {\cF}_e^{(s)}$, $\mu \in {\cP}$,
 ${\mathcal L}_{\mu}: {\cA}^{(s)} \rightarrow {\cA}^{(s)}$ and
$a_{\mu}^*:  [{\cO}(\G)]\rightarrow [{\cO}(\G)]$.

The following theorem can be regarded as a refined version of the
Global to Local Bijection of \cite{BMM}. The case $e=1$ is due to 
Enguehard (\cite{E}, p.37), where the proof is a direct verification
of the theorem from the definition of the Farahat map $\cF$ in 
${\mathcal S}_n$ (see Section \ref{Fock2}) and Lusztig induction in $G_n$.

\begin{thm} Under the isomorphism ${\cA}^{(s)} \cong [{\cO}(\G)]$
given by $[{\chi}_{\lam}] \rightarrow [{\Delta}_{\lambda^*}]$ where
$\lambda^*$ is the $e$-quotient of $\lambda$, 
 Lusztig induction ${\mathcal L}_{\mu}$ 
on ${\cA}^{(s)}$ with $\mu \in \cP$ corresponds to parabolic induction
$a_{\mu}^*$ on $[{\cO}(\G)]$ with $\mu \in {\cP}^e$.
\end{thm}

\begin{proof}


Consider the action of $b_{s_{\mu}} \in U(\fH)$ on ${\cF}_{e,e}^{(s)}$. The operator $S_{\mu}$ 
acting on ${\cF}_{e,e}^{(s)}$ can be identified 
with $a_{\mu}^*$  acting on $[{\cO}(\G)]$, with the basis element $|\lambda, s> $
 corresponding to  $[\Delta_{\lam}]$ (\cite{SV}, 5.20).
 
On the other hand,  $b_{s_{\mu}} \in U(\fH)$  acts as $S_{\mu}$ on the space ${\cF}_{e}^{(s)}$
and thus, by Theorem \ref{Lusztig} as ${\mathcal L}_{\mu}$ on ${\cA}^{(s)}$ with the 
basis element $|\lambda, s> $  corresponding to  $[\chi_{\lam}]$.
Here we note that Lusztig induction preserves $e$-cores, and thus ${\mathcal L}_{\mu}$ 
fixes ${\cA}^{(s)}$ .

Now ${\cF}_{e}^{(s)}$  is isomorphic to ${\cA}^{(s)}$  and
${\cF}_{e,e}^{(s)}$  is isomorphic to $[{\cO}(\G)]$. Thus 
we have shown that $a_{\mu}^*$ and ${\mathcal L}_{\mu}$  correspond
under two equivalent representations of $U(\fH)$. 

\end{proof}

\begin{cor}
 The BMM-bijection of Theorem \ref{BMM}  between the constituents of 
the Lusztig map $R^{G_n}_L(\lam)$ where $(L, \lam)$ is an $e$-cuspidal pair
and the characters of $W_{G_n}(L, \lam)$ 
is described via equivalent representations of $U(\fH)$ on Fock spaces.
\end{cor}
This follows from the theorem, using the map $\rm spe$ (see Section\ref{RDAHA}).

\section{Decomposition numbers}\label{decn}

References for this section are \cite{DJ}, \cite{LT1}, \cite{LT2}.
In this section we assume $G_n=GL(n,q)$, since we will be using the 
connection with $q$-Schur algebras. We describe connections between
weight spaces of $\widehat{{\mathfrak s \ell}_e}$ on Fock space,
blocks of $q$-Schur algebras, and blocks of $G_n$. We show that some
Brauer characters of $G_n$ can be described by Lusztig induction.

The $\ell$-decomposition numbers of the groups $G_n$ have been
studied by Dipper-James and by Geck, Gruber, Hiss and Malle. The latter 
have also studied the classical groups, using modular Harish-Chandra induction.
One of the key ideas in these papers is to compare the decomposition matrices of 
the groups with those of $q$-Schur algebras.

We have the Dipper-James theory over a field of characteristic $0$ or $\ell$.
They define (\ref{DJ}, 2.9) the $q$-Schur algebra  ${\mathcal S}_q(n)$,
endomorphism algebra of a sum of permutation representations
of the Hecke algebra $\mathcal{H}_n$ of type $A_{n-1}$ \cite{DJ}.
 The unipotent characters and the $\ell$-modular Brauer characters 
of $G_n$ are both indexed by partitions of $n$ (see \cite{FS}). Similarly  the 
Weyl modules and the simple modules of ${\mathcal S}_q(n)$ 
are both indexed by partitions of $n$ (see \cite{DJ}).

 For  ${\mathcal S}_q(n)$ over $k$ of characteristic $\ell$,  $q \in  k$,
one  can define the decomposition matrix of ${\mathcal S}_q(n)$,
where $q$ is an $e$-th root of unity, where as before $e$ is the order of
$q$ mod $\ell$. By the above this is a square matrix 
whose entries are the multiplicities of simple modules in Weyl modules.
Dipper-James (\cite{DJ}, 4.9) showed that this matrix, up to reordering the rows 
and columns, is the same as the unipotent part of 
the $\ell$-decomposition matrix of $G_n$,  the transition matrix
between the ordinary (complex) characters and the $\ell$-modular
Brauer characters. The rows and columns of the matrices are 
indexed by partitions of $n$.

We consider the Fock space ${\cF}= {{\cF}_e}^{(d)}$ for a fixed $d$, 
which as in Section \ref{Fock2} is isomorphic to $\mathcal A$,
and has the standard basis $\lbrace |\lam > \rbrace, \lam \in \cP$.
It also has two canonical bases
$G^{+}(\lambda)$ and $G^{-}(\lambda)$, $\lam \in \cP$ ({\cite{LT1},\cite{LT2}). 
There is a recursive algorithm to determine these two bases.

We fix an $s$ as in Section \ref{Fock2}).
The algebra $\widehat{{\mathfrak s \ell}_e}$ acts on ${\cF}_{e,e}^{(s)}$ 
and hence on ${\cF}_{e}^{(s)}$, which is a weight space for the
algebra. 
The connection between $\widehat{{\mathfrak s \ell}_e}$-weight spaces and 
blocks of the $q$-Schur algebras and hence blocks of $GL(n,q), n\geq 0$
is known, and we describe it below.
 We denote the Weyl module of ${\mathcal S}_q(n)$ 
parametrized by $\lam$ by $W(\lam)$.

We need to introduce a function $\rm res$ on $\cP$. If $\lam \in \cP$, 
The $e$-residue of the $(i,j)$-node of the Young diagram of $\lam$ 
 is the non-negative integer $r$
given by $r \equiv j-i ({\rm mod}\ e), \ 0 \leq r< e$, 
denoted ${\rm res}_{i,j}(\lam)$.
Then ${\rm res}(\lam)= \cup_{(i,j)}({\rm res}_{i,j}(\lam))$. 

\begin{prop} A weight space for $\widehat{{\mathfrak s \ell}_e}$
on ${\cF}_{e}^{(s)}$ can be regarded as a block of a $q$-Schur algebra with $q$
a primitive $e$-th root of unity.
\end{prop}
\begin{proof}
Two Weyl modules $W(\lam), W(\mu)$ are in the same block if and only if
${\rm res}(\lam)= {\rm res}(\mu) $ (see e.g.  (\cite{Ma}, Theorem 5.5, 
$(i) \Leftrightarrow (iv)$). 
The fact that \emph{res} defines a weight space follows e.g.
from (\cite{RSVV}, p.60).
\end{proof} 

Thus a weight space determines a set of partitions of a fixed $n \geq 0$.

\begin{cor} A weight space for $\widehat{{\mathfrak s \ell}_e}$
on ${\cF}_{e}^{(s)}$ can be regarded as a block of $GL(n,q)$,
where $n$ is determined from the weight space. 
\end{cor}

We now have the following theorem which connects the $\ell$-decomposition
numbers of $G_n, n\geq 0$ with Fock space.
 
\begin{thm} Let $\phi_{\mu}$ be the Brauer character of $G_n$ indexed by
$\mu \in {\cP}_n$. Let $\lam \in {\mathcal P}_n$.Then, for large $\ell$, 
$(\chi_{\mu}, \phi_{\lam}) = (G^{-}(\lambda),|\mu>)$, where
$(G^{-}(\lambda),|\mu>)$ is the coefficient of  $|\mu>)$ in the expansion of
the canonical basis in terms of the standard basis.
\end{thm}
\begin{proof}
 The decomposition matrix  of ${\mathcal S}_q(n)$ 
over a field of characteristic $0$, with $q$ a root of unity, is known
by Varagnolo-Vasserot \cite{VV}. By their work  the coefficients in the expansion of
the $G^{+}(\lambda)$ in terms of the standard basis give the
decomposition numbers for the algebras ${\mathcal S}_q(n)$ , $n \geq 0$,
with $q$ specialized at an $e$-th root of unity (see \cite{A}).

By an asymptotic argument of Geck \cite{G} we can pass from the decomposition matrices of 
$q$-Schur algebras in characteristic $0$ to those in characteristic $\ell$,
where $\ell$ is large. Then by the Dipper-James theorem we can pass to 
the decomposition matrices of the groups $G_n$ over a field of 
characteristic $\ell$ with $q$ an $e$-th root of unity in the field.

Let $D_n$ be the unipotent part of the $\ell$-decomposition matrix of $G_n$ and 
$E_n$ its inverse transpose.
Thus  $D_n$ has columns $G^{+}(\lambda)$ and $E_n$ has rows given
by $G^{-}(\lambda)$ (see \cite{LT1}, Section 4). The rows of $E_n$ also 
give the Brauer characters of
$G_n$, in terms of unipotent characters. These two descriptions of the rows of
$E_n$ then gives the result.
\end{proof}         

The following analog of Steinberg's Tensor Product Theorem is proved
for the canonical basis $G^{-}(\lambda)$ in \cite{LT1}.

\begin{thm}\label{Leclerc} Let $\lam$ be a partition such that ${\lam}'$ is
$e$-singular, so that  $\lam = \mu + e\alpha$ where $\mu'$ is $e$-regular. Then 
$G^{-}(\lambda)= S_{\alpha}G^{-}(\mu)$.
\end{thm}

We now show that the rows indexed by partitions $\lam$ as in the above theorem
can be described by Lusztig induction. By replacing $S_{\alpha}$ by 
${\mathcal L}_{\alpha}$ and using Theorem \ref{Lusztig} it follows that
in these cases, Lusztig-induced characters coincide with Brauer characters.          

\begin{thm}\label{Brauer} Let $\lam = \mu + e\alpha$ where $\mu'$ is $e$-regular.
Then the Brauer character represented by $G^{-}(\lambda)$ is equal to the Lusztig 
generalized character $R_L^{G_n}(G^{-}(\mu)\times \chi_{\alpha})$,
where $L= G_m \times GL(k,q^e)$, $n=m+ke, \alpha \vdash k$.
\end{thm}

By using the BMM bijection, Theorem \ref{BMM}, we have the following corollary.

\begin{cor} Let $\mu = \phi$, so that $\lam =  e\alpha$. 
Then the Brauer character represented by $G^{-}(\lambda)$ can be
calculated from an induced character in a complex reflection group.
\end{cor}

Examples, thanks to GAP \cite{Ga}:

Some tables giving the basis vectors $G^{-}(\lambda), e=2$ are given in \cite{LT2}. 
In our examples we use transpose partitions of the partitions in these tables,
and rows instead of columns.

We first give an example of a weight space for ${\widehat {sl}_e}$, which is 
also a block for $G_n$, with $n=4$, $e=4$. This is
an example of a decomposition matrix $D$ for $n=4$, $e=4$:

$\begin {pmatrix}4||&1&0&0&0\cr
             31||&1&1&0&0\cr
             211||&  0&1&1&0\cr
              1111||&  0&0&1&1\cr
                \end {pmatrix}$

The following example is to illustrate Theorem \ref{Brauer}.
                
                An example of the inverse of a decomposition matrix for
$n=6$, $e=2$: $\begin {pmatrix}1&0&0&0&0&0&0&0&0&0\cr
            -1&1&0&0&0&0&0&0&0&0\cr
            1&-1&1&0&0&0&0&0&0&0\cr
              -1&0&-1&1&0&0&0&0&0&0\cr
               -1&1&-1&0&1&0&0&0&0&0\cr
              1&-1&1&-1&-1&1&0&0&0&0\cr
              1&0&1&-1&-1&0&1&0&0&0\cr
              0&0&-1&1&1&-1&-1&1&0&0\cr
              0&0&1&-1&0&0&1&-1&1&0\cr
            0&0&0&0&0&0&-1&1&-1&1\cr
                \end {pmatrix}$

Here the rows are indexed as: $6, 51, 42,41^2,3^2,31^3,2^3,2^21^2,
21^4,1^6$

In the above matrix:

The rows indexed by $1^6, 2^21^2, 3^2, 21^4, 41^2$ have
interpretations as Brauer characters, in terms of $R_L^{G_n}$, 
with $L$ $e$-split Levi of
the form $GL(3,q^2)$ for $\lam = 1^6, 2^21^2, 3^2$, of the form
$GL(2,q)\times GL(2,q^2)$ for $\lam=21^4$, and of the form
$GL(4,q) \times GL(1,q^2)$ for $\lam=41^2$.

With $L=GL(3,q^2)$, 

Row indexed by $3^2$ is 
$R^G_L(\chi_{3})= \chi_{3^2}- \chi_{42}+ \chi_{51}- \chi_{6}$,

Row indexed by $2^21^2$ is $R^G_L(\chi_{21})$ and

Row indexed by $1^6$ is $R^G_L(\chi_{1^3})$.

\section{Acknowledgement}

The author thanks Bernard Leclerc for valuable discussions on 
this paper.


\begin{thebibliography}{10}

\bibitem{BMM}
M. Brou\'{e}, G. Malle \& J. Michel, Repr\'{e}sentations Unipotentes
g\'{e}n\'{e}riques et blocs des groupes r\'{e}ductifs finis,
Ast\'{e}risque 212 (1993).

\bibitem{BE} R.Bezrukavnikov and P.Etingof, Parabolic induction and restriction
functors for rational Cherednik algebras, Selecta math. 14 (2009), 397-425.

\bibitem{CE}
M. Cabanes and M. Enguehard, Representation theory of finite reductive groups, New Mathematical Monographs, 1, Cambridge University Press, Cambridge, 2004.

\bibitem{CT} J.Chuang \& K.M.Tan, Canonical basis of the basic
$U_q(\widehat {sl_n})$-module, J.Algebra 248 (2002), 765-779.

\bibitem{DJ} R.Dipper and G.James, The $q$-Schur algebra, \emph{Proc. London Math. Soc.} \textbf{59} (1989), 23-50.

\bibitem{E}
M. Enguehard. Combinatoire des Symboles et Application de Lusztig.
\emph{Publicationes du LMENS} \textbf{92-9} (1992).

\bibitem{F} H.Farahat, On the representations of the symmetric
group, Proc. London Math. Soc. 4 (1954), 303-316.

\bibitem{FS}
P.Fong and B.Srinivasan, The blocks of finite general linear and unitary groups, \emph{Invent. Math.} \textbf{69} (1982), 109--153.

\bibitem{Ga} The GAP Group: GAP—Groups, Algorithms, and Programming, 
Version 4.4.10. http://www.gap-system.org (2007)

\bibitem{G} M.Geck, Modular Harish-Chandra series, Hecke algebras and
(generalized) $q$-Schur algebras. {\it In}: Modular Representation
Theory of Finite Groups (Charlottesville, VA, 1998; eds. M.~J.~Collins,
B.~J.~Parshall and L.~L.~Scott), p.~1--66, Walter de Gruyter, Berlin 2001

  

\bibitem{GKS} F. Garvan, D. Kim and D.Stanton, Cranks and t-cores,
\emph{Invent. Math.} \textbf{1-01} (1990), 1-17.

\bibitem{L} T.Lam, Ribbon Tableaux and the Heisenberg algebra, \emph{Math.
Zeit.} \textbf{250} (2005), 685-710.

\bibitem{LLT} A.Lascoux, B.Leclerc, J-Y. Thibon, Ribbon tableaux,
Hall-Littlewood functions, quantum affine algebras and unipotent
varieties, \emph{J.Math.Phys.} \textbf{38} (1997), 447-456.

\bibitem{LT1} B.Leclerc, J-Y. Thibon, Canonical bases of q-deformed
Fock spaces, \emph{Int. math. Res. Notices} \textbf{9} (1996), 447-456.

\bibitem{LT2}  B.Leclerc, J-Y. Thibon, Littlewood-Richardson
coefficients and Kazhdan-Lusztig polynomials, Advanced Studies in
Pure Math. 28 (2000), 155-220.

\bibitem{M} I.G.Macdonald, \emph{Symmetric functions and Hall polynomials},
Oxford (1995).

\bibitem{Ma} A.Mathas, \emph{The representation Theory of the Ariki-Koike
and cyclotomic $q$-Schur  algebras}, Adv. Stud. Pure Math. \textbf{40,}
261-320, Math. Soc. Japan, Tokyo, 2004. 

\bibitem{RSVV} R.Rouquier, P.Shan, M.Varagnolo, E.Vasserot, Categorifications and
cyclotomic rational double affine Hecke algebras, arXiv: 1305.4456v1

\bibitem{SV} P.Shan, E.Vasserot, Heisenberg algebras and rational double affine Hecke
algebras, \emph{J. Amer. Math. Soc.} \textbf{25} (2012),
959-1031.

\bibitem{U} D. Uglov, Canonical bases of higher level q-deformed Fock spaces and Kazhdan-Lusztig poly-
nomials, \emph{Progress in Math}. \textbf{191}, Birkhauser (2000).

\bibitem{VV} M.Varagnolo, E.Vasserot, On the decomposition numbers of the
quantized Schur algebra, Duke Math. J. 100 (1999), 267-297.
\end{thebibliography}
\end{document}